\documentclass{amsart}
\usepackage{amsmath, amssymb,epic,graphicx,mathrsfs,enumerate}
\usepackage[all]{xy}
\usepackage{color}
\usepackage{comment}
\usepackage{tikz}

\usepackage{amsthm}
\usepackage{amssymb}
\usepackage{latexsym}
\usepackage{longtable}
\usepackage{epsfig}
\usepackage{amsmath}
\usepackage{hhline}

\DeclareMathOperator{\Hom}{Hom}

\DeclareMathOperator{\alt}{Alt}
 
\DeclareMathOperator{\der}{Der}

\newtheorem{thm}{Theorem}

 \newtheorem{lemma}[thm]{Lemma}

\numberwithin{equation}{section}

\renewcommand{\footnote}{\endnote}
\newcommand{\ignore}[1]{}\makeglossary

\begin{document}
	\title[Generating graph]{Finite groups with planar generating graph}

	\author{Andrea Lucchini}
	\address{Andrea Lucchini\\ Universit\`a degli Studi di Padova\\  Dipartimento di Matematica \lq\lq Tullio Levi-Civita\rq\rq\\ Via Trieste 63, 35121 Padova, Italy\\email: lucchini@math.unipd.it}

\begin{abstract} Given a finite  group $G$,  the generating graph  $\Gamma(G)$ of $G$ has as  vertices the  non-identity elements of $G$ and two vertices  are adjacent if and only if they are distinct and generate $G$ as group elements. Let $G$ be a 2-generated finite group. We prove that $\Gamma(G)$ is planar if and only if 	$G$ is isomorphic to one of the following groups: $C_2, C_3, C_4, C_5, C_6, C_2 \times C_2, D_3, D_4, Q_8, C_4\times C_2,  D_6.$	\end{abstract}
	\maketitle
	
\section{Introduction}

Given a finite  group $G$,  the generating graph  $\Gamma(G)$ of $G$ has as  vertices the non-identity elements of $G$ and two vertices  are adjacent if and only if they are distinct and generate $G$ as group elements. 

When $G$ is simple  many deep results on generation of $G$ in the literature can be translated to results about $\Gamma(G)$. For example, the  property   that $G$ can be generated by two elements amounts to saying that $\Gamma(G)$ has at least one edge. The fact due to  Guralnick and Kantor in \cite{GK} that every nontrivial element of $G$ belongs to a generating pair of elements of $G$ is equivalent to saying that $\Gamma(G)$ has no isolated vertices.   More recently, Breuer, Guralnick and Kantor  proved in \cite{BGK} that $G$ has spread at least 2, or in other words $\Gamma(G)$ has diameter at most 2.

More generally, one can try to characterise finite groups $G$ for which a given graph-theoretical property holds in  $\Gamma(G)$. 

As an illustration, recall that a graph $\Gamma$ is Hamiltonian (respectively, Eulerian) if it contains a cycle going through every vertex (respectively, edge) of $\Gamma$ exactly once. 
In \cite{BGLMN},  Breuer, Guralnick, Mar\'{o}ti, Nagy and the first author have investigated  the finite groups $G$ for which $\Gamma(G)$ is Hamiltonian. For example they showed that every finite simple group  of large enough order has an Hamiltonian generating graph. In \cite{marion}, Marion and the first author have studied generating graphs for the alternating and symmetric groups on $n$ points proving that they are Eulerian if and only if $n$ and $n-1$ are
not a prime congruent to 3 modulo 4.

The aim of this note it to determine the finite groups $G$ with the property that $\Gamma(G)$ is a planar graph. Recall that a graph is said to be embeddable in the plane, or planar, if it can be drawn in the plane so that its edges intersect only at their ends. If $G$ cannot be generated with two elements, then $\Gamma(G)$ contains no edge and all the vertices are isolated. So we may restrict our attention to the finite groups that can be generated by two elements. We prove that there are only finitely many 2-generated finite groups $G$ such that $\Gamma(G)$ is planar. More precisely we have:

\begin{thm}\label{main}
	Let $G$ be a finite 2-generated  group. Then $\Gamma(G)$ is planar if and only if
	$G\in \{C_2, C_3, C_4, C_5, C_6, C_2\times C_2, D_4, Q_8, C_4\times C_2, D_3, D_6\},$
where, as usual, with $C_n$ we denote the cyclic group of order $n$, with $D_n$ the dihedral group of order 2n and with $Q_8$  the quaternion group.\end{thm}

We will prove the previous theorem showing that, with only finitely many exceptions, the number $e(G)$ of edges of $\Gamma(G)$ is at least $3|G|.$ Some crucial preliminary results in this direction will be proved in Section \ref{prel}, where we will compare
the ratios  $e(G)/|G|$ and $e(G/N)/|G/N|$ in the particular case when $N$ is a minimal normal subgroup of $G.$

\section{Preliminary results}\label{prel}
Let $N$ be a normal subgroup of a finite group $G$ and choose $g_1,\dots,g_k\in G$ with the property that $G=\langle g_1,\dots,g_k\rangle N.$ By a result of Gasch{\"u}tz \cite{g1} the cardinality of the set $$\Phi_N(g_1,\dots,g_n)=\{(n_1,\dots,n_k)\in N\mid \langle g_1n_1,\dots,g_kn_k=G\rangle \}$$ does not depend on the choice of $g_1,\dots,g_k.$ Let
$$P_{G,N}(k)=\frac{|\Phi_N(g_1,\dots,g_n)|}{|N|^k}.$$ 
Denote by $P_X(k)$ the probability that $k$ randomly chosen elements of $X$ generate the finite group $X$. Notice that if $X$ is a finite 2-generated group, then the number $e(X)$ of the edges of $\Gamma(X)$ is $|X|^2P_G(X)/2.$
It turns out that $P_G(k)=P_{G/N}(k)P_{G,N}(k)$, and $P_{G,N}(k)$ is the conditional probability that $k$-randomly chosen elements of $G$ generate $G$ given that they generate $G$ modulo $N$.

For the remainder of this section we will assume that $G$ is a 2-generated group finite and  $N$ is a minimal normal subgroup of $G$ and we will define $$\alpha(G,N):=\frac{e(G)/|G|}{e(G/N)/|G/N|}=|N|P_{G,N}(2).$$

\begin{lemma}\label{uno}If $N$ is non-abelian, then $\alpha(G,N)\geq 1.$
	\end{lemma}
\begin{proof}
Assume $G=\langle g_1,g_2\rangle$ and let $n\in N.$ We have $$G=\langle g_1^n, g_2^n\rangle=\langle g_1[g_1,n],g_2[g_2,n]\rangle,$$ hence $([g_1,n],[g_2,n])\in \Phi_N(g_1,g_2).$ On the other hand, if $n_1, n_2\in N$ and
 $$([g_1,n_1],[g_2,n_1])= ([g_1,n_2],[g_2,n_2])),$$ then $n_1n_2^{-1}\in C_G(g_1)\cap C_G(g_2)=Z(G).$ Since $N$ is non-abelian, it must be $n_1n_2^{-1}\in N\cap Z(G)=1,$ so $n_1=n_2.$ We deduce that $|\Phi_N(g_1,g_2)|\geq |N|$, and consequently $\alpha(G,N)\geq 1.$
 \end{proof}

\begin{lemma}\label{due}
	If $N$ is non-abelian and $G/N$ is soluble, then $\alpha(G,N)> 35.$
\end{lemma}
\begin{proof}
	It follows from \cite[Theorem 17]{crowns} that if $G/N$ is soluble, then $$P_{G,N}(2)=P_{G/C_G(N),NC_G(N)/C_G(N)}(2).$$ Moreover
	by \cite[Theorem 1.1]{lon},
$$P_{G/C_G(N),NC_G(N)/C_G(N)}(2)\geq \frac{53}{90}.$$ Since $|N|\geq |\alt(5)|\geq 60,$
we conclude $$\alpha(G,N)=P_{G,N}(2)|N|\geq \frac{53}{90}\cdot 60>35.\qedhere$$
\end{proof}

\begin{lemma}\label{tre}
	Assume that $N$ is abelian. We have $|N|=p^a,$ where $p$ is a prime and $a$ is a positive integer. Let $c$ be the number of complements of $N$ in $G$.
	Then $$\alpha(G,N)=\frac{p^{2a}-c}{p^a}\geq p^a-p^{a-1}.$$
	In particular
	\begin{enumerate}
		\item $\alpha(G,N)=1$ if and only if $|N|=2$, $N$ has a complement in $G$ 
		and $G/N$ has an epimorphic image of order 2.
		\item $\alpha(G,N)=3/2$ if and only if $|N|=2,$  $N$ has a complement in $G$ 
		and no epimorphic image of $G/N$ has  order 2.
		\item $\alpha(G,N)\geq 2$ in all the remaining cases.
	\end{enumerate}
\end{lemma}
\begin{proof}By \cite[Satz 2]{g2}, $P_{G,N}(2)=1-c/p^{2a},$ hence $\alpha(G,N)=\frac{p^{2a}-c}{p^a}$. If $c\neq 0,$ then $c$ is the order of the group $\der(G/N,N)$  of derivations from $G/N$ to $N$; in particular $c$ is a power of $p$. Moreover, since $G$ is 2-generated, it must be $c<p^{2a}$ and consequently
	$$\alpha(G,N)=\frac{p^{2a}-c}{p^a}\geq \frac{p^{2a}-p^{2a-1}}{p^a}=p^a-p^{a-1}.$$ In particular we can have $\alpha(G,N)<2$ only if $|N|=2$ and $c\neq 0.$
	Let $H$ be  a complement of $N$ in $G$ and let $K=H^\prime H^2.$ We have $c=|\der(H,N)|=|\Hom(H/K,N)|.$ Since $G$ is 2-generated, either $H/K=1$ or $H/K\cong C_2$. In the first case $c=1$ and $\alpha(G,N)=3/2,$ in the second case $c=2$ and $\alpha(G,N)=1.$
	\end{proof}

\section{Proof of Theorem \ref{main}}
	
Our proof of Theorem \ref{main} will rely on the following result in graph theory.

\begin{thm}{\cite[Corollary 10.21]{bm}}\label{crit} A planar graph with $n\geq 3$ vertices has at most $3n-6$ edges.
\end{thm}

\begin{proof}[Proof of Theorem \ref{main}]First we prove that if $G\in \{C_2, C_3, C_4, C_5, C_6, C_2\times C_2, D_3, D_4,$ $Q_8, C_4\times C_2, D_6\},$ then $\Gamma(G)$ is planar. Let $\Delta(G)$ be the subgraph of $\Gamma(G)$ obtained by removing the isolated  vertices. Clearly $\Gamma(G)$ is planar if and only if $\Delta(G)$ is planar. We have $\Delta(C_2)=K_1$,  $\Delta(C_3)=K_2$, $\Delta(C_4)=K_3$, $\Delta(C_5)=K_4$	and $\Delta(C_2\times C_2)=K_3$ (as usual, we denote by $K_n$ the complete graph on $n$ vertices). 
If $G=\langle g \rangle$ is cyclic of order $6,$ then $\Gamma(G)=\Delta(G)$ can be  drawn as follows:
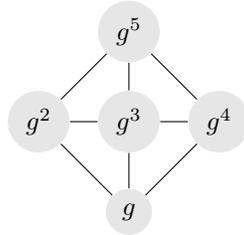
\begin{center}
	\begin{tikzpicture}
	[scale=.6,auto=left,every node/.style={circle,fill=black!10}]
	\node (n4) at (5,9)  {$g^5$};
	\node (n1) at (3,7) {$g^2$};
	\node (n2) at (5,7)  {$g^3$};
	\node (n3) at (5,5)  {$g$};
	\node (n5) at (7,7)  {$g^4$};
	\foreach \from/\to in {n1/n3,n1/n2,n2/n3,n1/n4,n2/n4,n3/n5,n4/n5,n2/n5}
	\draw (\from) -- (\to);
	;
	\end{tikzpicture}
\end{center}
If $G$ is a non-cyclic group of order 8, then $\Delta(G)$ has 12 edges and 6 vertices and is isomorphic to $\Delta(Q_8)$, which is planar as indicated by the following picture.
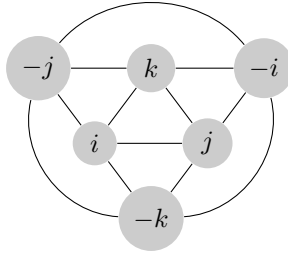
\begin{center}
	\begin{tikzpicture}
	[scale=.5,auto=left,every node/.style={circle,fill=black!20}]
	\node (n6) at (2,9) {$-j$};
	\node (n4) at (5,9)  {$k$};
	\node (n5) at (8,9)  {$-i$};
	\node (n1) at (3.5,7) {$i$};
	\node (n2) at (6.5,7)  {$j$};
	\node (n3) at (5,5)  {$-k$};
	\foreach \from/\to in {n1/n3,n1/n2,n2/n3,n1/n4,n2/n4,n4/n5,n2/n5,n4/n6,n1/n6}
	\draw (\from) -- (\to);
	\path
	(n5) edge[bend right=50] (n6)
	(n6) edge[bend right=50] (n3)
	(n5) edge[bend left=50] (n3)
	;
	\end{tikzpicture}
\end{center}
Let $G=D_3=\langle a, b \mid a^3, b^2, abab \rangle$. Then $\Delta(G)$ can be drawn as follows:
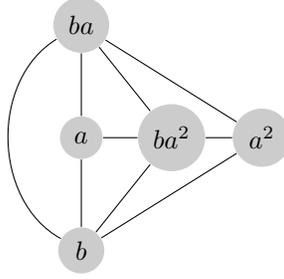
\begin{center}
	\begin{tikzpicture}
	[scale=.6,auto=left,every node/.style={circle,fill=black!20}]
	\node (n1) at (5,2.5)  {$b$};
	\node (n2) at (5,7.5) {$ba$};
	\node (n3) at (7,5)  {$ba^2$};
	\node (n4) at (5,5)  {$a$};
	\node (n5) at (9,5) {$a^2$};
	\foreach \from/\to in {n1/n3,n2/n3,n1/n4,n2/n4,n1/n5,n2/n5,n4/n3,n5/n3}
	\draw (\from) -- (\to);
	\path
(n1) edge[bend left=60] (n2);
	\end{tikzpicture}
\end{center}
Let $G=D_6=\langle a, b \mid a^3, b^2, abab \rangle$. Then $\Delta(G)$ can be drawn as follows:
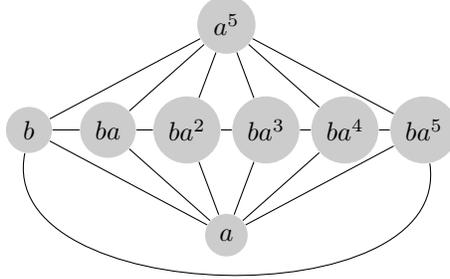
\begin{center}
	\begin{tikzpicture}
	[scale=.7,auto=left,every node/.style={circle,fill=black!20}]
	\node (n1) at (1,5)  {$b$};
	\node (n2) at (2.5,5) {$ba$};
	\node (n3) at (4,5)  {$ba^2$};
	\node (n4) at (5.5,5)  {$ba^3$};
	\node (n5) at (7,5) {$ba^4$};
	\node (n6) at (8.5,5) {$ba^5$};
	\node (n7) at (4.75,3) {$a$};
	\node (n8) at (4.75,7) {$a^5$};
	\foreach \from/\to in {n1/n2,n2/n3,n3/n4,n4/n5,n5/n6,n1/n7,n2/n7,n3/n7,n4/n7,n5/n7,n6/n7,n1/n8,n2/n8,n3/n8,n4/n8,n5/n8,n6/n8}
	\draw (\from) -- (\to);
	\path
	(n1) edge[bend right=100] (n6);
	\end{tikzpicture}
\end{center}

Assume now that $G$ is a 2-generated finite group and let $e(G)$ be the number of vertices of the generating graph $\Gamma(G).$	We have $$e(G)=\frac{|G|^2P_G(2)}{2},$$ denoting by $P_G(k)$ the probability that $k$ randomly chosen elements generate $G.$

Assume that $\Gamma(G)$ is a planar graph. Then, by Theorem \ref{crit},
$$\frac{|G|^2P_G(2)}{2}\leq 3|G|-6.$$
In particular 
\begin{equation}\label{main0}
|G|P_G(2)<6.
\end{equation}

Let  $1=N_t\leq \dots \leq N_0=G$ be a chief series of $G$. We have
$$P_G(2)=\prod_{1\leq i \leq t}P_{G/N_i,N_{i-1}/N_{i}}(2) \quad \text{and} \quad |G|=\prod_{1\leq i \leq t}|N_{i-1}/N_i|.$$
So, setting $\alpha_i=|N_{i-1}/N_i|P_{G/N_i,N_{i-1}/N_{i}}(2)=\alpha(G/N_i,N_{i-1}/N_i)$ for $1\leq i\leq t,$ we deduce from (\ref{main0}) that
\begin{equation}\label{impo}
\prod_{1\leq i\leq t}\alpha_i < 6.
\end{equation}
If follows from Lemmas \ref{uno} and \ref{tre}, that $\alpha_i\geq 1$ for every $1\leq i \leq t.$ Hence we deduce from (\ref{impo}) that $\alpha_i<6$ for  every $1\leq i \leq t.$ This implies that $G$ is soluble. Otherwise we could find $j$ such that $G/N_{j-1}$ is soluble and $N_{j-1}/N_j$ is non-abelian, and therefore $\alpha_j>35$ by Lemma \ref{due}.

Assume that $G$ is cyclic of order $n\geq 7.$ Then $\phi(n)\geq 4,$ so there exist four different elements $g_1,g_2,g_3,g_4$ with $G=\langle g_i\rangle$ for
$1\leq i \leq 4.$ Choose $x\in G\setminus \{1,g_1,g_2,g_3,g_4\}.$ The subgraph of $\Gamma(G)$ induced on the subset $\{g_1,g_2,g_3,g_4,x\}$ is isomorphic to $K_5,$ so $\Gamma(G)$ is not planar. We may so assume that $G$ is not cyclic.

Firstly, assume that $C_2\times C_2$ is an epimorphic image of $G.$ It is not restrictive to assume $N_0/N_2\cong C_2 \times C_2$. Lemma \ref{tre} implies $\alpha_1=3/2$, $\alpha_2=1$ and $\alpha_3\geq 2$ if $j>2$ (notice that if $j>2$ and $N_{j-1}/N_j$ has order 2, then $N_{j-1}/N_j$ cannot be complemented in $G/N_j$, otherwise $C_2^3$ would be an epimorphic image of $G$ and $G$ could not be generated by 2 elements). By (\ref{impo}), we must have $t\leq 3$. If $t=2,$ then $G\cong C_2\times C_2.$ If $t=3$, then $\alpha_3 <4$ and, again by Lemma \ref{tre}, $|N_2|\leq 4$. If $|N_2|=2,$ then $G$ is a non-cyclic 2-generated group of order 8, i.e. $G\in \{C_4\times C_2, D_4, Q_8\}$. The possibility $|N_2|=4$ cannot occur: $G$ would be a group of order 16 and it could not contain a minimal normal subgroup of order $4$.
Assume $|N_2|=3.$ If $N_2$ is non-central in $G$, then $G\cong D_6$. If $N_2\leq Z(G)$, then $G\cong C_2\times C_2\times C_3$: in this case $\Delta(G)$ has 9 vertices and 24 edges: since $24>3 \cdot 9-6=21$,
$\Delta(G)$ is not planar by Theorem \ref{crit}.

Finally assume that $C_2\times C_2$ is not an epimorphic images of $G$. In this case, again by Lemma \ref{tre},  $\alpha_1\geq 3/2$ and $\alpha_j\geq 2$ if $j > 1$. 
 By (\ref{impo}), we must have $t\leq 2$, and consequently $t=2$ since we are assuming that $G$ is not cyclic. By Lemma  \ref{tre}, $|N_1|\leq 4.$ Since $G$ is not cyclic, we remain with the following possibilities: $G\cong C_3\times C_3,$ $G \cong \alt(4)$, $G\cong D_3$. The first two cases can be excluded by Theorem \ref{crit}: $\Delta(\alt(4))$ has 48 edges and 11 vertices, $\Delta(C_3\times C_3)$ has 24 edges and 9 vertices.
\end{proof}

\end{document}